\documentclass[a4paper,11pt]{article}
\usepackage[ngerman, english]{babel}
\usepackage[T1]{fontenc}
\usepackage[utf8]{inputenc}
\usepackage{amsmath}
\usepackage{amssymb}
\usepackage{amsthm}
\usepackage{graphicx}
\usepackage{here}
\usepackage{color}
\usepackage{xcolor}
\usepackage{enumerate}
\usepackage{lmodern}
\usepackage{fancyvrb}
\usepackage[plainpages=false]{hyperref}
\usepackage{caption}
\usepackage{subfigure}
\usepackage{epstopdf}
\newcommand{\ca}{\mathcal}

\newtheorem{defi}{Definition}[section]
\newtheorem{theo}[defi]{Theorem}
\newtheorem{prop}[defi]{Proposition}
\newtheorem{lem}[defi]{Lemma}
\newtheorem{con}[defi]{Conjecture}

\theoremstyle{remark}

\newcommand{\ENDproof}{\hfill $\square$\medskip\par}

\title{Edge-connectivity and pairwise disjoint perfect matchings in regular graphs}

\author{Yulai Ma$^1$\thanks{Supported by Sino-German (CSC-DAAD) Postdoc Scholarship Program 2021 (57575640)} \footnotemark[3], Davide Mattiolo$^2$\thanks{Supported by a Postdoctoral Fellowship of the Research Foundation Flanders (FWO), project number 1268323N}, Eckhard Steffen$^1$, Isaak H.~Wolf$^1$\thanks{Funded by Deutsche Forschungsgemeinschaft (DFG, German Research Foundation) - 445863039} \\
	\footnotesize
	$^1$ Department of Mathematics, Paderborn University, Warburger Str.\ 100, 33098 Paderborn,
	Germany.
	\\
	\footnotesize
	$^2$ Department of Computer Science, KU Leuven Kulak, 8500 Kortrijk, Belgium.
	\\ \footnotesize yulai.ma@upb.de, davide.mattiolo@kuleuven.be, es@upb.de, isaak.wolf@upb.de}

\date{}

\begin{document}

\maketitle

\begin{abstract} 
For $0 \leq t \leq r$ let $m(t,r)$ be the maximum number $s$ such that every $t$-edge-connected $r$-graph has $s$ pairwise disjoint perfect matchings. There are only a few values of $m(t,r)$ known, for instance $m(3,3)=m(4,r)=1$, and $m(t,r) \leq r-2$ for all $t \not = 5$, and  $m(t,r) \leq r-3$ if $r$ is even. We prove that $m(2l,r) \leq 3l - 6$ for every $l \geq 3$ and $r \geq 2 l$. 
\end{abstract}


\section{Introduction and motivation}
This paper studies regular graphs which may have parallel edges but no loops. All graphs and all multisets considered in this paper are finite.
The vertex set of a graph $G$ is denoted by $V(G)$ and its edge set by $E(G)$. The number of parallel edges connecting two vertices $u,v$ of $G$ is denoted by $\mu_G(u,v)$ and $\mu(G)=\max\{\mu_G(u,v)\colon u,v\in V(G)\}$. For any two disjoint subsets $ X $ and $ Y $ of $ V(G) $  we denote by $E_G(X,Y)$ the set of edges with one end in $X$ and the other end in $Y$. If $Y= X^c$ we denote $E_G(X,Y)$ by $\partial_G(X).$ The graph induced by $ X $ is denoted by $ G[X] $.
The \emph{edge-connectivity} of $G$, denoted by  $\lambda(G)$, is  the maximum number $k$ such that
$|\partial_G(X)|\ge k$ for every non-empty  $X \subset V(G)$. Similarly, the \emph{odd edge-connectivity} $\lambda_o(G)$ is defined as the maximum number $t$ such  that
$|\partial_G(Y)|\ge t$ for every  $Y \subseteq V(G)$ of odd cardinality.
Clearly, $\lambda(G) \leq \lambda_o(G)$ if $ G $ is of even order. 
An $r$-regular graph $G$ is an $r$-\emph{graph} if $\lambda_o(G) = r$.
Note that an $r$-graph can have small edge-cuts which separate sets of
even cardinality.  
 
An $r$-graph is \emph{class 1}, if its edge set can be partitioned into
$r$ perfect matchings and it is \emph{class 2} otherwise. Consequently, an $r$-graph is
class 2 if and only if it has at most $r-2$ pairwise disjoint perfect matchings.  There are many hard problems with regard to perfect matchings in $r$-graphs. For instance Seymour's exact conjecture \cite{seymour1979multi}
states that every planar $r$-graph is class 1. 

If $G$ has a set of $k$ pairwise disjoint perfect matchings we say that it has a
$k$-PDPM. For $0 \leq t \leq r$ let $m(t,r)$ be the maximum number $s$ such that every $t$-edge-connected $r$-graph has an $s$-PDPM. 
In addition to its exact determination, lower and upper bounds for this parameter are of great interest. 
The function $m(t, r)$ is monotone increasing in $t$, in other words $m(t, r)\leq m(t', r )$ for $t\leq t' $. In particular we have that
$m(t,r) \leq m(r,r)$ for all $t \in \{2, \dots,r\}$. 

For all $r \geq 3$ and $r \not = 5$, class 2 $r$-edge-connected $r$-graphs 
are known, \cite{MMSW_pdpm, Meredith}. Thus,
$m(r,r) \leq r-2$ for these $r$. Surprisingly, no such graphs seem to
be known for $r=5$. 
The question whether $m(5,5)=5$ is raised in \cite{MMSW_pdpm} where also some
consequences (if true) are discussed.  
Thomassen \cite{THOMASSEN2020343} asked whether $m(r,r) = r-2$. This is not
true if $r$ is even. 
In \cite{MMSW_pdpm, Mattiolo2022HighlyER} it is proved that
$m(r,r) \leq r-3$ if $r$ is even and $m(r-1,r) \leq r-3$ if $r$ is odd. 

There is also not much known with regard to lower bounds for $m(t,r)$.
Trivially we have $m(t,r) \geq 1$. It is an open question 
which $r$-graphs have two disjoint perfect matchings.  
Rizzi \cite{rizzi1999indecomposable} constructed
(non-planar) $r$-graphs where any two perfect matchings intersect.
These $r$-graphs are called poorly matchable and so far all known
poorly matchable $r$-graphs have a 4-edge-cut. 
We do not know any 5-edge-connected poorly matchable $r$-graph. It might be 
that high edge-connectivity (instead of odd edge-connectivity) enforces the existence of pairwise disjoint perfect matchings in $r$-graphs. 
Thomassen \cite{THOMASSEN2020343} conjectured this to be true. Precisely, he conjectured that there is an integer $r_0$
such that there is no poorly matchable $r$-graph for $r \geq r_0$.
Seymour's exact conjecture implies that there is 
no poorly matchable planar $r$-graph. However, even this seemingly weaker
statement is unproved so far.  
Up to now, we have $m(3,3)=1$ for cubic graphs and $m(4,r)=1$ by Rizzi's result. 

In this paper, we improve upper bounds for $m(t,r)$ which only depend on 
the edge-connectivity  parameter. 
Our main result is that $m(2l,r) \leq 3l - 6$ for every $l \geq 3$ and $r \geq 2l$.

\section{Basic definitions and results}

In this paper, we make extensive use of the Petersen graph, denoted by $P$, and of the properties of its perfect matchings. Rizzi \cite{rizzi1999indecomposable} observed that every two distinct 1-factors of the Petersen graph have precisely one edge in common, and  proved that there is a one-to-one
correspondence between edges and pairs of distinct $ 1 $-factors in the Petersen graph.   Then we have the following proposition immediately.

\begin{prop}\label{Prop-Petersen-six-matching}
	The Petersen graph has exactly six perfect matchings, and each edge is contained in exactly two of them.
\end{prop}

We fix a drawing of $P$ as in Figure \ref{fig:Pet_plus_pm} left. With reference to Figure \ref{fig:Pet_plus_pm}, we define $M_0$ to be the perfect matching consisting of all edges $u_iv_i$, for $i\in\{1,\dots,5\}$. Moreover, for $i\in\{1,\dots,5\}$, by Proposition \ref{Prop-Petersen-six-matching} we let $M_i$ be the only other perfect matching of $P$ different from $M_0$ and containing $u_iv_i$, see Figure \ref{fig:Pet_plus_pm}.

Let $G$ be a graph and let $F_1,\dots,F_t\subseteq E(G)$. The graph $G+F_1+\ldots+F_t$ is the graph obtained from $G$ by adding a copy of every edge in $F_j$ for every $j\in\{1,\dots,t\}$. 
For a multiset $\ca N$ of perfect matchings of $G$ and an edge $ e\in E(G) $, we say that $\ca N$ \emph{contains} (\emph{avoids}, respectively) $e$  if $e\in \bigcup_{N\in \ca N} N$  ($e\notin \bigcup_{N\in \ca N} N$, respectively).

Let $\ca M$ be a multiset of perfect matchings of $P$. We denote by $n_{\ca M}(i)$ the number of copies of $M_i$ appearing in $\ca M$.  
Define $P^{\ca M}$ to be the graph $P+\sum_{F\in\ca M}F$.
Now, let $\ca N$ be a multiset of perfect matchings of $P^{\ca M}$. 
Note that each perfect matching  of $ \mathcal{N} $ can be interpreted as a perfect matching of $P$ by  caring only about the end-vertices of each edge. Then the multiset  $ \ca N $ can be interpreted as a multiset of perfect matchings of $ P $, which is denoted by $\ca N_P $. Note that $ |\mathcal{N}_P | = |\mathcal{N}| $.

\begin{figure}
	\centering
	\includegraphics[scale=0.7]{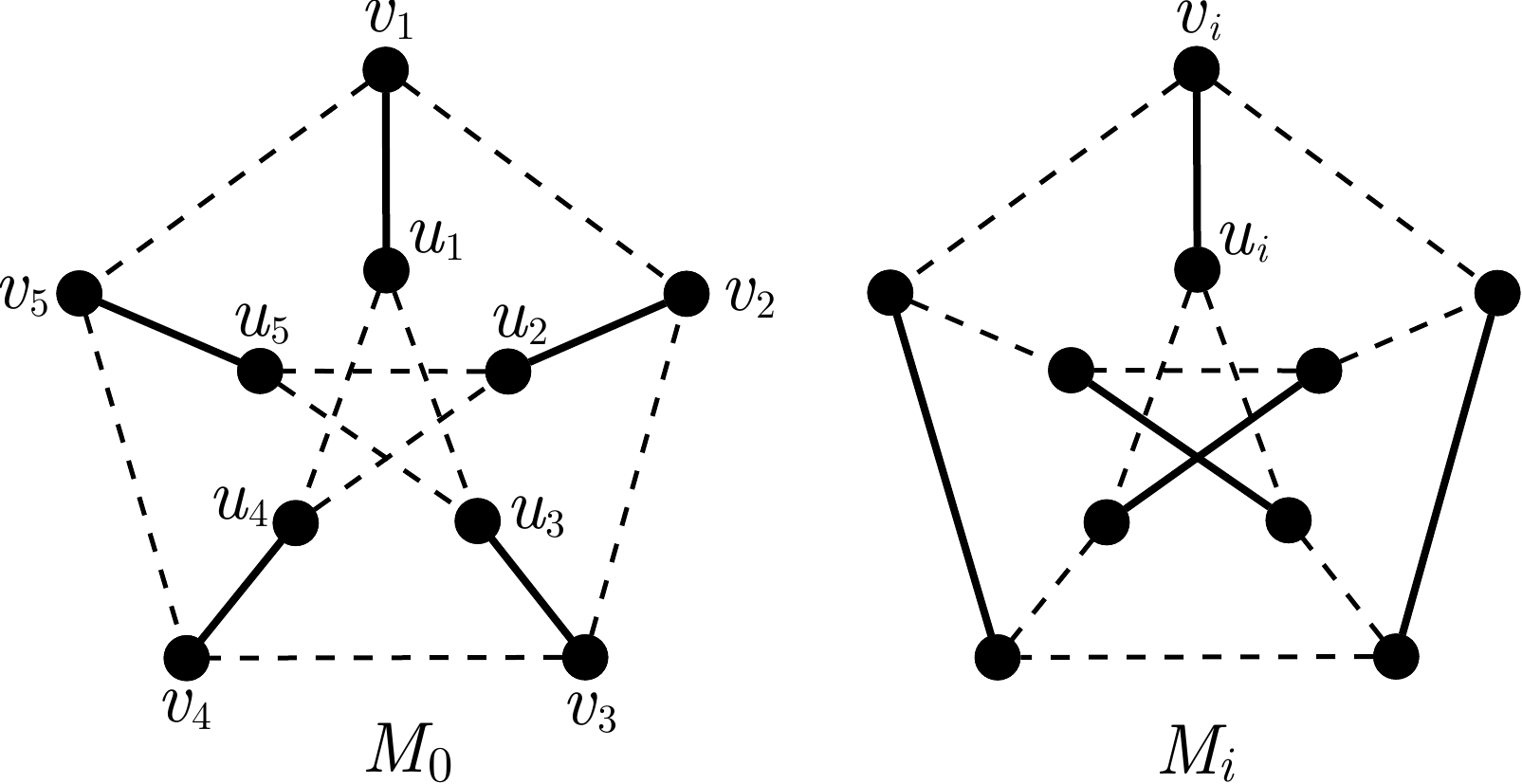}
	\caption{The Petersen graph $P$, and its perfect matchings $M_0$ and $M_i$.}\label{fig:Pet_plus_pm}
\end{figure}

\begin{lem}\label{Lemma-no-cover-all-edges}
	Let $\ca M$ be a multiset of perfect matchings of $P$. Let $\ca N$ be a set of pairwise disjoint perfect matchings of $P^{\ca M}$. There is at most one $i\in\{0,\dots,5\}$ such that $n_{\ca N_P}(i)>n_{\ca M}(i)$.

	In particular, there is no triple of different vertices $u,v,w$ in $P^{\ca M}$, with $w$ adjacent to both $v$ and $u$, such that $\ca N$ contains all edges of $E_{P^{\ca M}}(\{u,v\},\{w\})$.
\end{lem}

\begin{proof}
	First, suppose that there are two  indices $i$ and $j$ such that $i\ne j$, $n_{\ca N_P}(i)>n_{\ca M}(i)$, and $n_{\ca N_P}(j)>n_{\ca M}(j)$. Let $uv$ be the edge of $P$ belonging to both $M_i$ and $M_j$ by Proposition \ref{Prop-Petersen-six-matching}. 
	Since the perfect matchings of $ {\ca N} $  are pairwise disjoint,  at most $\mu_{P^{\ca M}}(u,v)$ perfect matchings in $ {\ca N} $ can contain an edge connecting $ u $ and $ v $. This implies $n_{\ca N_P}(i) + n_{\ca N_P}(j)\leq \mu_{P^{\ca M}}(u,v)$.
	Then the following contradiction arises. \[\mu_{P^{\ca M}}(u,v) = n_{\ca M}(i) + n_{\ca M}(j) + 1 \le n_{\ca N_P}(i) + n_{\ca M}(j) <  n_{\ca N_P}(i) + n_{\ca N_P}(j).\]
	
	Next, we prove the second part of the lemma. Let $u,v$ be two different vertices both adjacent to the vertex $w$ in $P^{\ca M}$. Suppose by contradiction that $\ca N$ contains all edges of $E_{P^{\ca M}}(\{u,v\},\{w\})$. By Proposition \ref{Prop-Petersen-six-matching}, we may assume without loss of generality that $\{uw\} = M_0\cap M_1$ and $\{vw\}= M_2\cap M_3.$ Then, since all edges of $E_{P^{\ca M}}(\{u,v\},\{w\})$ are contained in $\ca N$, we similarly deduce  that
	\begin{itemize}
		\item $n_{\ca N_P}(0)+n_{\ca N_P}(1) =\mu_{P^{\ca M}}(u,w)= n_{\ca M}(0)+n_{\ca M}(1)+1$;
		\item $n_{\ca N_P}(2)+n_{\ca N_P}(3) =\mu_{P^{\ca M}}(v,w)= n_{\ca M}(2)+n_{\ca M}(3)+1$.
	\end{itemize} Then we conclude that there is $s\in\{0,1\}$ and $t \in \{2,3\}$, such that $n_{\ca N_P}(s)> n_{\ca M}(s)$ and $n_{\ca N_P}(t)> n_{\ca M}(t)$, which is impossible.
\end{proof}

\begin{lem}\label{Lemma-edge-conn-Peterson}
	Let $\ca M$ be a multiset of $k$ perfect matchings of $P$ and let $\mu=\mu(P^{\ca M})$. 
	Then, $\lambda (P^{\ca M})=\min\{k+3,2k+6-2\mu\}.$
\end{lem}

\begin{proof}
	Note that $P$ is a $3$-graph, and $P^{\ca M}$ is a $(k+3)$-graph since every perfect matching of $ P $ intersects each edge-cut  that separates two vertex sets of odd cardinality.
	Let $X$ be a non-empty proper subset of $ V(P^{\ca M})$ minimizing $|\partial_{P^{\ca M}}(X)|$. It implies that $ P^{\ca M}[X] $ is connected. If $ |X| $ is odd, then $|\partial_{P^{\ca M}}(X)|\ge k+3$ since	$P^{\ca M}$ is a $(k+3)$-graph. If $ |X| $ is even, then it suffices to consider the cases $ |X|\in \{2,4\} $. Since $P$ has girth $5$, either $P[X]$  is a path on two or four vertices, or it is isomorphic to $ K_{1,3} $. Then $\partial_{P^{\ca M}}(X)$ contains at least $k+3-\mu$ edges for each vertex of degree 1 in $P[X]$, and so we get that $|\partial_{P^{\ca M}}(X)|\ge 2(k+3-\mu)$. Consequently, $\lambda (P^{\ca M})\geq\min\{k+3,2k+6-2\mu\}.$
	Finally, observe that $ |\partial_{P^{\ca M}}(\{u,v\})|=2k+6-2\mu$ for every two vertices $ u,v $ with $ \mu_{P^{\ca M}}(u,v)=\mu $. Thus, the statement follows.
\end{proof}

\begin{defi}\label{Definition-t-splicing}
	Let $G$ and $H$ be two graphs with $u,v \in V(G)$ and $x,y \in V(H)$ such that $\mu_ {G}(u,v)\geq t$ and $\mu_ {H}(x,y)\geq r-t$. Then, $(G,u,v)\oplus_t (H,x,y)$ is the graph obtained from $G$ and $H$ by deleting exactly $t$ edges joining $u$ and $v$ in $G$ and $r-t$ edges joining $x$ and $y$ in $H$,  identifying $u$ and $x$ to  a new vertex $w_{ux}$, and  identifying $v$ and $y$ to a new vertex $w_{vy}$, see Figure~\ref{fig:t-splicing}. 
\end{defi}

\begin{figure}[htb]
	\centering
	\resizebox{0.7\textwidth}{!}{
		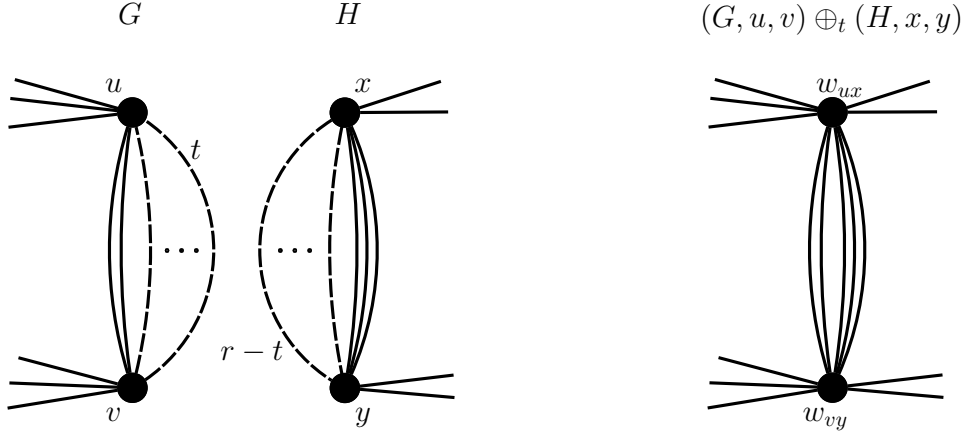}
	\caption{The operation of Definition \ref{Definition-t-splicing}.}
	\label{fig:t-splicing}
\end{figure}

\begin{lem}\label{Lemma-t-splicing-edge-conn}
	Let $G$ and $H$ be two $r$-graphs with $u,v \in V(G)$ and $x,y \in V(H)$ such that $\mu_ {G}(u,v)\geq t$ and $\mu_ {H}(x,y)\geq r-t$. Then, $G'=(G,u,v)\oplus_t (H,x,y)$ is an $r$-graph with $$\lambda(G')=\min\{\lambda(G),\lambda(H)\}.$$
\end{lem}

\begin{proof}
By the construction in Definition \ref{Definition-t-splicing}, $G'=(G,u,v)\oplus_t (H,x,y)$ is $r$-regular. 
For every $Y\subseteq V(G) \setminus \{v\}$, we have $|\partial_G(Y)|=|\partial_{G'}((Y\setminus\{u\})\cup\{w_{ux}\})|$ if $u \in Y$ and $|\partial_G(Y)|=|\partial_{G'}(Y)|$ otherwise. Similarly, for every $Y'\subseteq V(H) \setminus \{y\}$, we have $|\partial_H(Y')|=|\partial_{G'}((Y'\setminus\{x\})\cup\{w_{ux}\})|$ if $x \in Y'$ and $|\partial_H(Y')|=|\partial_{G'}(Y')|$ otherwise.
As a consequence, there is an $X \subseteq V(G')$ with $|\partial _{G'}(X)|=\min \{\lambda(G),\lambda(H)\}$.
Thus, it suffices to prove that, for each non-empty proper subset $S\subset V(G')$, we have $|\partial_{G'}(S)|\geq \min\{\lambda(G),\lambda(H)\}$ if $|S|$ is even, and  $|\partial_{G'}(S)|\geq r$ if $|S|$ is odd.
Since, for all $Y\subseteq V(G'), \partial_{G'}(Y)=\partial_{G'}(V(G')\setminus Y)$, we just need to  consider the two cases when $|S\cap\{w_{ux},w_{vy}\}|=0$ and $|S\cap\{w_{ux},w_{vy}\}|=1$.

First, assume  $|S\cap\{w_{ux},w_{vy}\}|=0$.  It is clear that $|\partial_{G'}(S)|= |\partial_{G}(S\cap V(G))|+|\partial_{H}(S\cap V(H))|$. Note that one of $|S\cap V(G)|$ and $|S\cap V(H)|$ is odd if $|S|$ is odd. So $|\partial_{G'}(S)|\geq \min \{\lambda(G),\lambda(H)\}$ if $|S|$ is even, and  $|\partial_{G'}(S)|\geq r$ if $|S|$ is odd. 

Next, we assume $|S\cap\{w_{ux},w_{vy}\}|=1$. Without loss of generality, say $w_{ux} \in S$. Note that 
$|\partial_{G'}(S)|= |\partial_{G}((S\setminus\{w_{ux} \}\cup \{u \})\cap V(G))|-t+|\partial_{H}((S\setminus\{w_{ux} \}\cup \{x \})\cap V(H))|-(r-t)$. If one of $|(S\setminus\{w_{ux} \}\cup \{u \})\cap V(G)|$ and $|(S\setminus\{w_{ux} \}\cup \{x \})\cap V(H)|$ is odd, then the other has the same parity as $|S|$. This implies $|\partial_{G'}(S)|\geq \min \{\lambda(G),\lambda(H)\}$ if $|S|$ is even, and  $|\partial_{G'}(S)|\geq r$ if $|S|$ is odd. Thus, the remaining case is that both $|(S\setminus\{w_{ux} \}\cup \{u \})\cap V(G)|$ and $|(S\setminus\{w_{ux} \}\cup \{x \})\cap V(H)|$ are even, and $|S|$ is odd.  
Since  $|(S\setminus\{w_{ux} \})\cap V(G)|$ is odd in this case and $G$ is an $r$-graph, we obtain $|\partial_{G}((S\setminus\{w_{ux} \}\cup \{u \})\cap V(G))|\geq 2 \mu_ {G}(u,v)\geq 2t$. Similarly,  $ |\partial_{H}((S\setminus\{w_{ux} \}\cup \{x \})\cap V(H))|\geq 2\mu_ {H}(x,y)\geq2(r-t)$. So $|\partial_{G'}(S)|\geq 2t-t+2(r-t)-(r-t)=r$.
 This completes the proof.
\end{proof}

\begin{lem}\label{lem:kSumWithP}
	Let $r,t$ be two integers with $2 \leq t < r$, let $G$ be an $r$-graph and let $u,v \in V(G)$ such that $\mu_G(u,v) \geq t$. Let $\ca M$ be a multiset of $r-3$ perfect matchings of $P$, let $x,y \in V(P^{\ca M})$ such that $\mu_{P^{\ca M}}(x,y)\geq r-t$ and let $G'=(G,u,v) \oplus_{t} (P^{\ca M},x,y)$. If $G'$ has a $k$-PDPM $\ca N'$, then $G$ has a $k$-PDPM $\ca N$ such that
	\begin{itemize}
		\item[$(i)$] $\ca N$ avoids at least one edge connecting $u$ and $v$,
		\item[$(ii)$] for every $e \in E(G'[V(G)\setminus\{u,v\}])$, if $\ca N'$ avoids $e$, then $\ca N$ avoids $e$.
	\end{itemize}
\end{lem}

\begin{proof}
	Assume that $\ca N'$ is a $k$-PDPM of $G'$. Every perfect matching of $G'$ contains either zero or exactly two edges of $\partial_{G'}(V(P^{\ca M})\setminus\{x,y\})$, since $|V(P^{\ca M})\setminus\{x,y\}|$ is even. The same holds for $ V (G) \setminus\{u,v\}$, since $|V (G) \setminus\{u,v\}| $ is also even.  Hence, every perfect matching of $ G' $ can  be transformed into a perfect matching of $ G $ and of $ P^\mathcal{M} $ by adding either $ uv $ or $ xy $. In particular, $\ca N'$ can be transformed into a $k$-PDPM $\ca N$ of $G$, which satisfies $(ii)$. Suppose that $\ca N$ contains all edges connecting $u$ and $v$, which implies that $\ca N'$ contains all edges of $\partial_{G'}(V(G))$. As a consequence, $P^{\ca M}$ has a $k$-PDPM that contains all edges of $\partial_{P^{\ca M}}(\{x,y\})$. This means that $P^{\ca M}$ has a $k$-PDPM containing all edges incident with $y$ and not with $x$, a contradiction to Lemma \ref{Lemma-no-cover-all-edges}.
\end{proof}

\section{An upper bound for $m(t,r)$ depending on $t$}

Recall that $m(t,r)\le m(t',r)$ whenever $t\le t'$. For an $r$-graph $G$ with a subset $ X\subseteq V(G) $,  we observe that $ |\partial_G(X)|=r\cdot |X|-2|E(G[X])|$ is even  if $ |X| $ is even.
Therefore, the edge-connectivity of an $ r $-graph is either $r$ or an even number. By Rizzi \cite{rizzi1999indecomposable}, $m(4,r)=1$ for every $r\geq 4$. Furthermore,  $r-2$ is a trivial upper bound for $m(2l,r)$ since $  m(4,5)=1$ and for each $r \neq 5$ there are $r$-edge-connected $r$-graphs that are class 2 \cite{MMSW_pdpm, Meredith}. We will improve this bound as follows.

\begin{theo}\label{theo:main_result}
	For every $l \geq 3$ and $ r\geq2l $, $m(2l,r) \leq 3 l - 6$.
\end{theo}

As mentioned above, we know $m(2l,r)\leq r-2 $ for every $ r\geq 3 $. It implies that Theorem \ref{theo:main_result} trivially holds for the case $ 2l\leq r \leq3l-4 $. Thus,  it suffices to prove Theorem \ref{theo:main_result} for the case $ r\geq 3l-3 $.

We will construct $2l$-edge-connected $r$-graphs inductively
starting with a $2l$-edge-connected $(3l-4)$-graph 
without a $(3l-5)$-PDPM if $l \geq 4$ and a
$6$-edge-connected $6$-graph without a $4$-PDPM if $l = 3$.

For this we describe the induction step in the next section and then in the following sections we give the base graphs
for the two cases. Finally we deduce the statement of Theorem \ref{theo:main_result}.

\subsubsection*{Induction step from $r$ to $r+1$}

\begin{lem}\label{lem:construction}
	Let $r, l, k$ be integers such that $r \geq 3l-4$, $l \geq 2$ and $2\le k\le r$. If there is an $r$-graph $G$ such that
	\begin{itemize}
		\item $\lambda(G) \geq 2l$,
		\item $G$ has a perfect matching $M$ such that $\mu_G(u,v)\ge l-1$ for every $uv\in M$,
		\item $G$ has no $k$-PDPM,
	\end{itemize}
	then there is an $(r+1)$-graph $G'$ such that
	\begin{itemize}
		\item $\lambda(G') \geq 2l$,
		\item $G'$ has a perfect matching $M'$ such that $\mu_{G'}(u,v)\ge l-1$ for every $uv\in M'$,
		\item $G'$ has no $k$-PDPM.
	\end{itemize}
\end{lem}

\begin{proof}
Assume that the order of $ G$  is  $2s $ and let $M=\{x_1y_1, \ldots, x_sy_s\}$. In order to construct $G'$ we define a graph $P_{(r+1,l)}$ by
	\begin{align*}
	P_{(r+1,l)}= P  + \left \lceil \frac{r-l}{2} \right \rceil M_0 + \left \lfloor \frac{r-l}{2} \right \rfloor M_1 + (l-2)M_2.
	\end{align*}
	Since $G$ is $2l$-edge-connected, we have $r\geq2l$. Thus, $P_{(r+1,l)}$ is well defined. For every $i \in \{1, \ldots, s\}$, take a copy $P_{(r+1,l)}^i$ of $P_{(r+1,l)}$. In each copy, the vertices and perfect matchings are labelled accordingly by using an upper index, i.e.\ the vertex of $P_{(r+1,l)}^i$ corresponding to $u_1$ in $P_{(r+1,l)}$ is labeled as $u_1^i$. Define graphs $H^0, \ldots, H^s$ inductively as follows:
	\begin{align*}
	&H^0:=G+M, \\ &H^{i}:= (H^{i-1},x_i,y_i) \oplus_{l} (P_{(r+1,l)}^i,u_1^i,v_1^i) \text{ for every } i \in \{1,\ldots,s\}.
	\end{align*}
	Note that $H^0$ and $P_{(r+1,l)}$ are both $(r+1)$-graphs. Furthermore, $\mu_{H^0}(x_i,y_i) \geq l$ for every $i \in \{1,\ldots,s\}$ by the choice of $M$. Recall that $u_1v_1 \in E(P)$ is the unique edge in $M_0 \cap M_1$. Thus, $\mu_{P_{(r+1,l)}}(u_1,v_1) = (r+1)-l$ by the definition of $P_{(r+1,l)}$. As a consequence, $H^0, \ldots ,H^s$ are well defined. Set
	\begin{align*}
	G':=H^s \qquad \text{ and } \qquad M':= \bigcup\limits_{i=1}^s M_2^i.
	\end{align*}
	
\begin{figure}[htbp]
\subfigure{
\begin{minipage}[t]{6cm}
\centering
\includegraphics[scale=0.3]{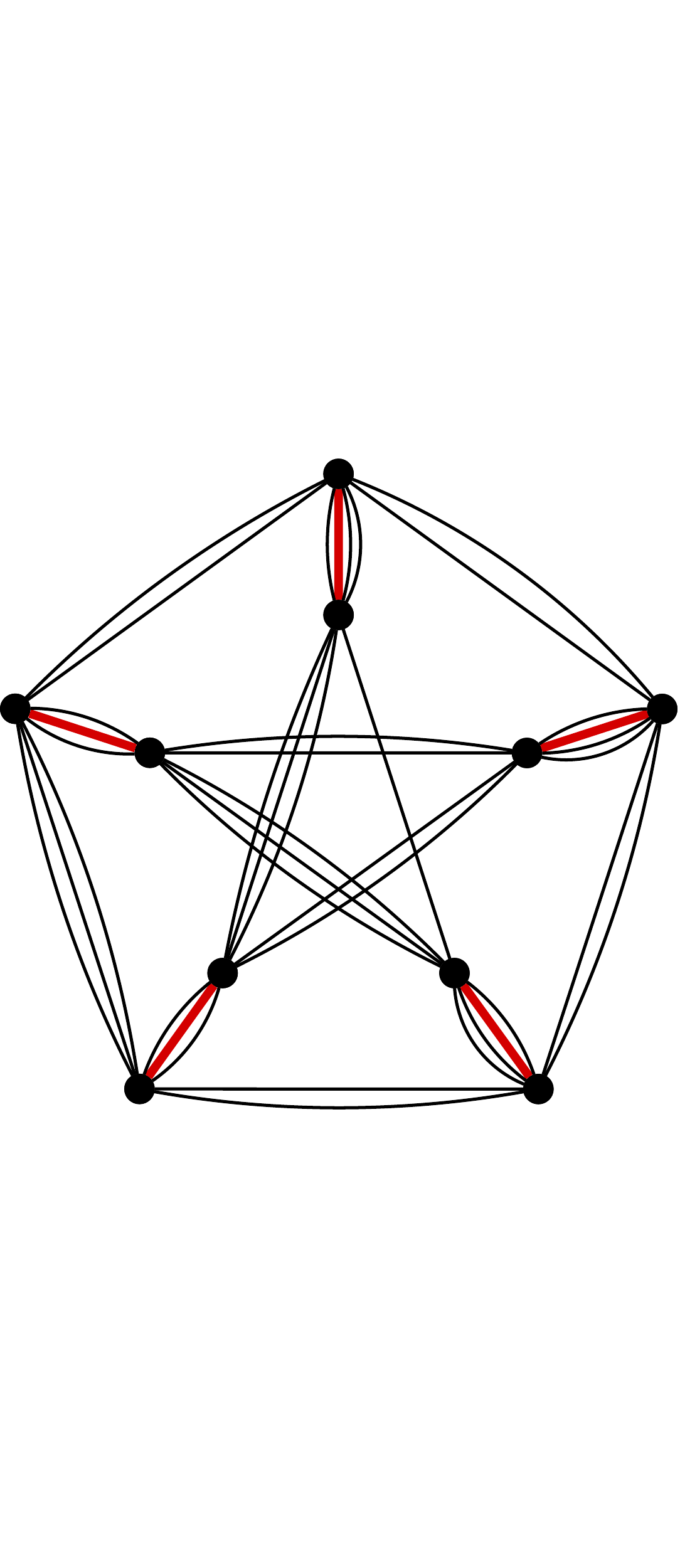}
\end{minipage}
}
\subfigure{
\begin{minipage}[t]{9.5cm}
\centering
\includegraphics[scale=0.3]{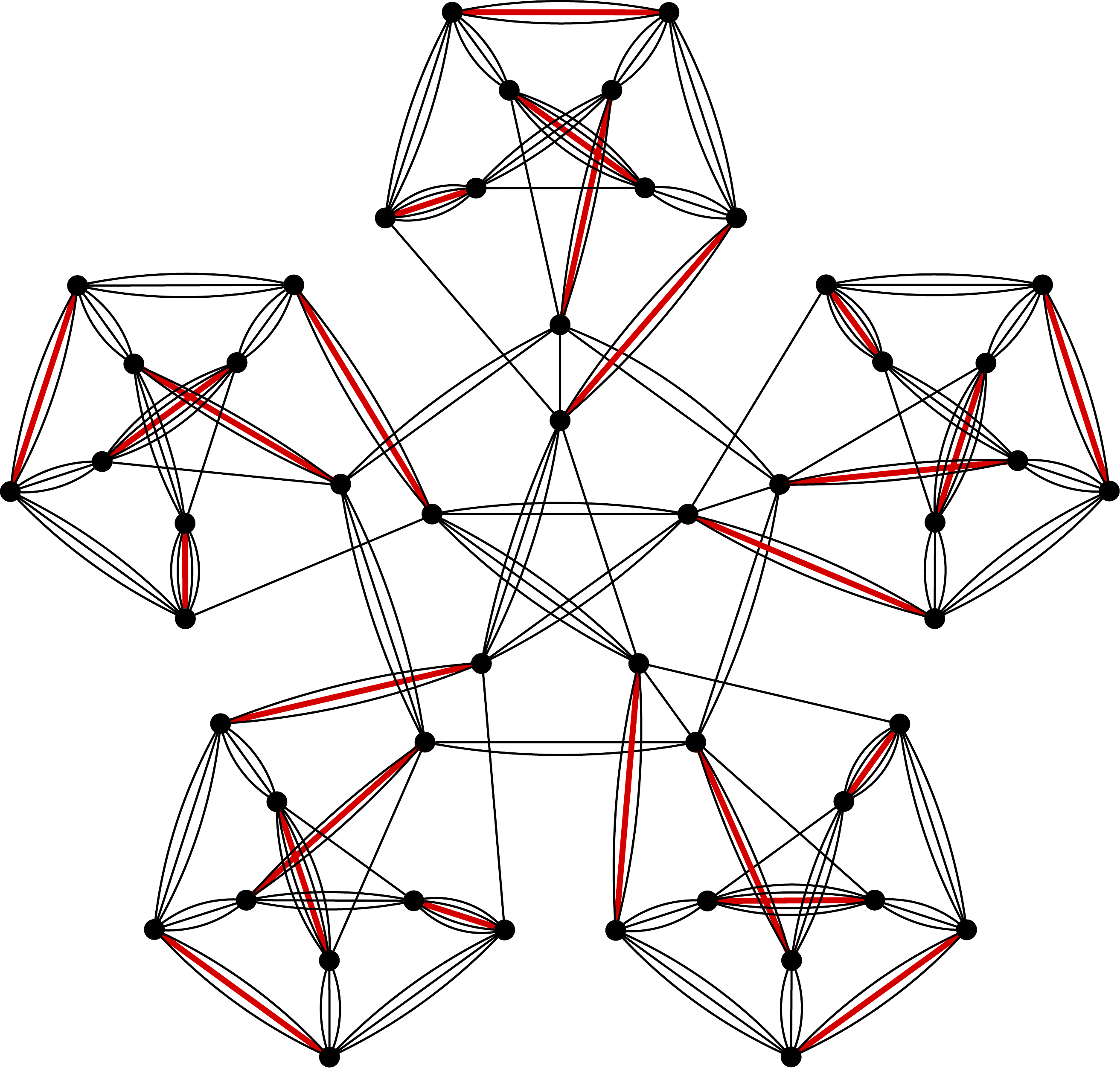}
\end{minipage}
}
\caption{The graph $G=P+2M_0+M_1+M_2+M_3$ (left) and the graph $G'$ (right) constructed from $G$ in the proof of Lemma~\ref{lem:construction}. The edges of $M$ and $M'$ respectively are drawn in bold red lines.}
\label{fig:G_8_G_9}
\end{figure}
	
	An example is given in Figure \ref{fig:G_8_G_9}. We claim that $G'$ and $M'$ have the desired properties.
	
	The perfect matching $M_2$ does not contain the edge $u_1v_1$. Thus, $M'$ is well defined. Furthermore, $M'$ is a perfect matching of $G'$ since $M$ is a perfect matching of $G$. By the definition of $P_{(r+1,l)}$, we have $\mu_{G'}(u,v)\geq l -1$ for every $uv \in M'$. Hence, $M'$ has the desired properties.
	
	The graph $H^0$ is a $2l$-edge-connected $(r+1)$-graph, since $G$ is a $2l$-edge-connected $r$-graph. Furthermore, $\left \lfloor \frac{r-l}{2} \right \rfloor \geq l -2$ since $r \geq 3l-4$. Thus, $r-l+1$ is the maximum number of parallel edges of $P_{(r+1,l)}$ and hence, $\lambda(P_{(r+1,l)})=2l$ by Lemma \ref{Lemma-edge-conn-Peterson}. Therefore, for each $ i\in\{1,\ldots,s\} $, $H^i$ is a $2l$-edge-connected $(r+1)$-graph by Lemma \ref{Lemma-t-splicing-edge-conn}, and so is $ G' $.
	
	Now, suppose that $H^s$ has a $k$-PDPM $\ca N^s$. By applying Lemma~\ref{lem:kSumWithP} with $ t=l $ to the $ (r+1) $-graph $H^s$ and $\ca N^s$ we obtain a $k$-PDPM $\ca N^{s-1}$ of $H^{s-1}$, which avoids $x^sy^s$ by property $(i)$. Apply Lemma~\ref{lem:kSumWithP} to $H^{s-1}$ and $\ca N^{s-1}$ to obtain a $k$-PDPM $\ca N^{s-2}$ of $H^{s-2}$, which avoids $x^{s-1}y^{s-1}$ by property $(i)$ and $x^sy^s$ by property $(ii)$. By inductively repeating this process, we obtain a $k$-PDPM of $H^0$ that avoids every edge of $M$. This is not possible, since $G$ has no  $k$-PDPM. Therefore, $G'$ has no  $k$-PDPM, which completes the proof.
\end{proof}

We note that the condition $r \geq 3l-4$ is necessary in Lemma~\ref{lem:construction} since $\lambda(P_{(r+1,l)})<2l$ if $r<3l-4$.
	In view of Lemma \ref{lem:construction}, we need to construct suitable base graphs for all $l\ge3$, which will be done now.

\subsubsection*{Base graph if $l = 3$.}

Let $l=3$ and let $P_1^1$ and $P_1^2$ be two copies of the graph $P+M_0+M_1+M_2$. For $i\in\{1,2\}$, remove from $P_1^i$ all parallel edges connecting $u_1^iv_1^i$, call this new graph $P_{-}^i$. Let $Q_1$ be the graph constructed by identifying the vertices $u_1^1$ and $u_1^2$ of $P_{-}^1$ and $P_{-}^2$ respectively. If a graph $G$ contains $Q_1$ as a subgraph, then let $E_1^i=E_G(\{v_1^i\}, V(G) \setminus V(Q_1))$ for every $i\in \{1,2\}$, see Figure \ref{fig:Q_1}. 

\begin{figure}[htbp]
	\centering
	\includegraphics[scale=0.4]{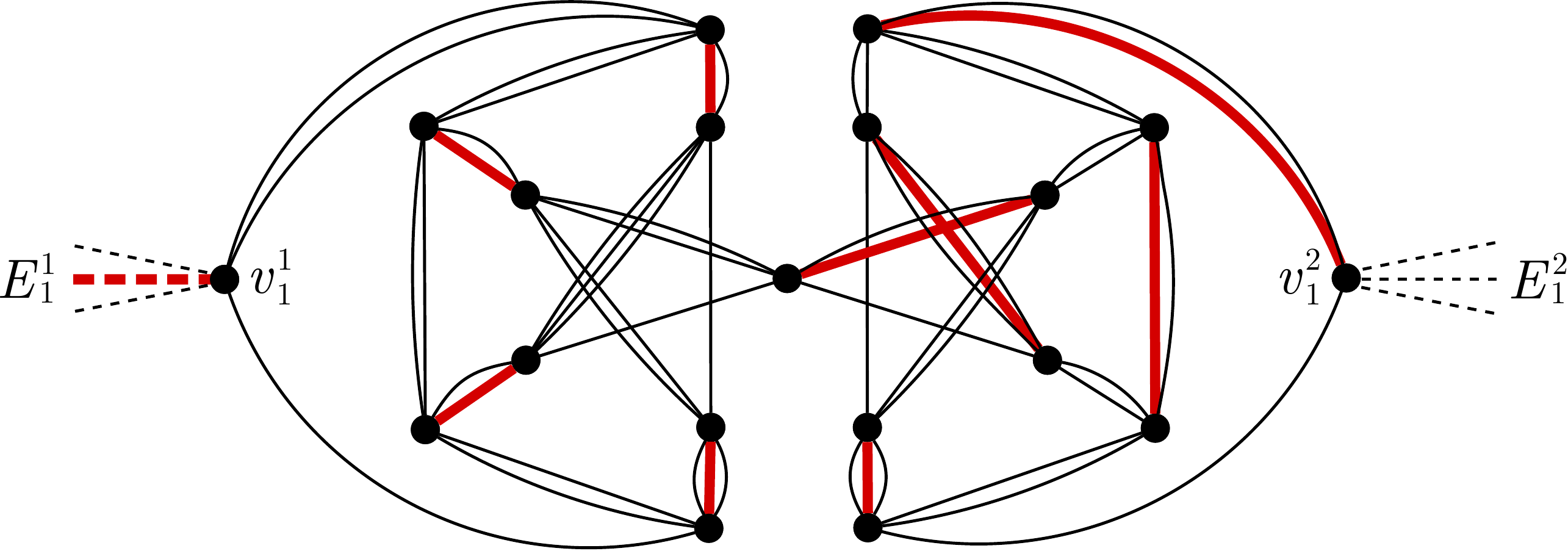}
	\caption{The graph $Q_1$ (solid lines) and the edge sets $E_1^1$, $E_1^2$ (dashed lines). The bold red edges are used to construct $M^6$ in the proof of Theorem~\ref{theo:main_result}.}\label{fig:Q_1}
\end{figure}

Recall the following property of $Q_1$, proved in \cite{MMSW_pdpm}.

\begin{lem}[\cite{MMSW_pdpm}]\label{lem:Q_1}
	Let $G$ be a graph that contains $Q_1$ as an induced subgraph. Let $\{N_1,\dots, N_{4}\}$ be a set of pairwise disjoint perfect matchings of $G$ and let $N= \bigcup_{i=1}^{4} N_i$. If $\partial(V(Q_1)) = E_1^{1} \cup E_1^{2}$, then $\vert E_1^{1} \cap N \vert= \vert E_1^{2} \cap N \vert =2.$
\end{lem}

In order to construct the required base graph $G^6$, we need the graph $G_1$ shown in Figure \ref{fig:G_1}, where the boxes denote copies of the graph $Q_1$.

\begin{figure}[htbp]
	\centering
	\includegraphics[scale=0.4]{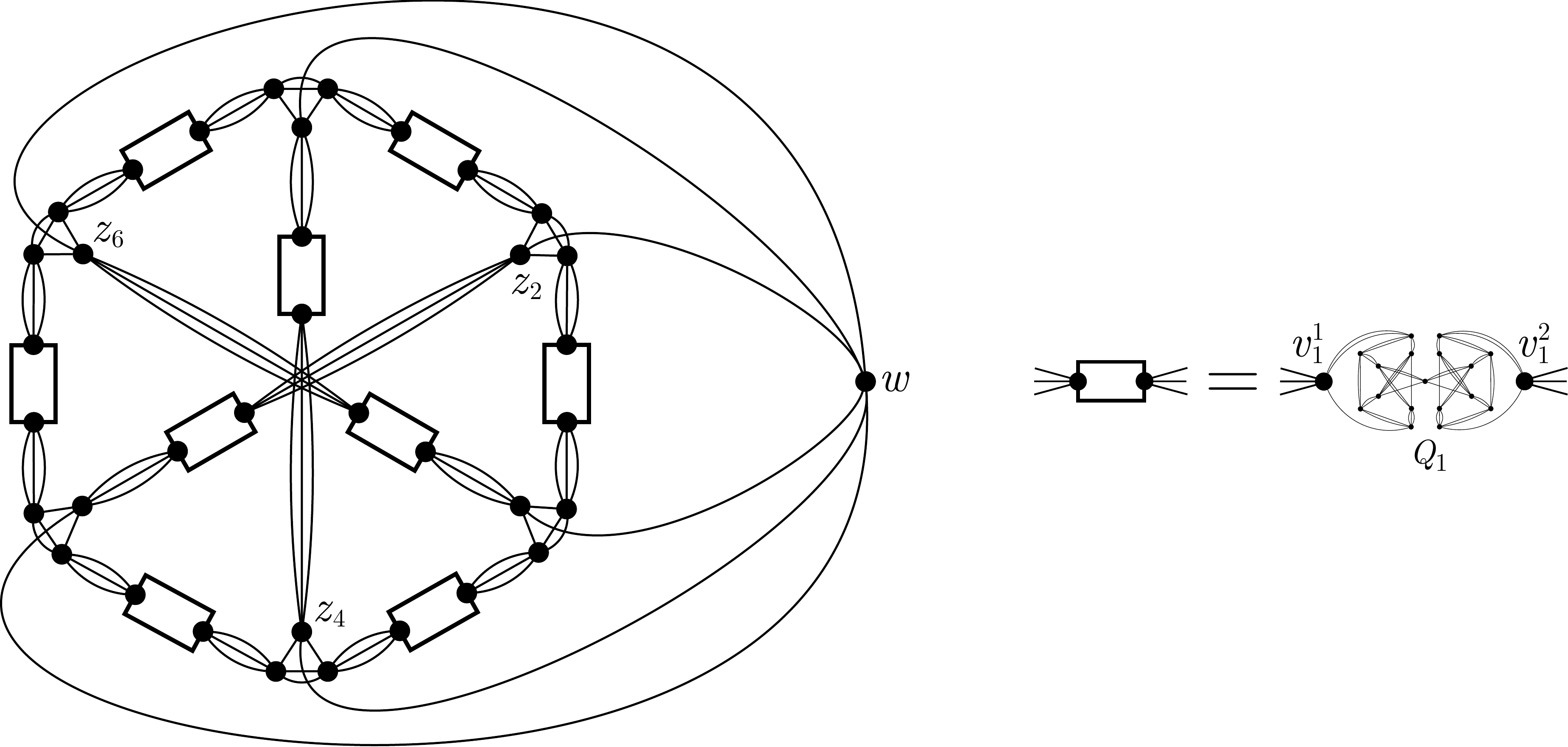}
	\caption{The graph $G_1$.}\label{fig:G_1}
\end{figure}

The graph $G_1$ was constructed in \cite{MMSW_pdpm} to provide a negative answer to a question of Thomassen \cite{THOMASSEN2020343}. The following theorem describes its properties.

\begin{theo}[\cite{MMSW_pdpm}]
	$G_1$ is a $6$-edge-connected $6$-graph without a $4$-PDPM.
\end{theo}
 Every perfect matching of $G_1$ contains an edge in $ \partial_{G_1}(w) $, which is simple.
 Thus, in order to use Lemma~\ref{lem:construction} we need to slightly modify $G_1$. For any $v\in V(G_1)$, we define a $3$-\emph{expansion} to be the operation that splits $v$ into two vertices $v'$ and $v''$ (edges formerly incident with $v$ will be incident with exactly one of $v'$ and $v''$) and adds three parallel edges between them.

Let $G^6$ be the graph (depicted in Figure \ref{fig:G^6}) obtained from $ G_1 $ by applying a $3$-expansion to the vertices $z_2,z_4,z_6$ and $w$. Let $w'$ and $w''$ be the new vertices in which $w$ has been split. It is straightforward that $G^6$ is still a $6$-edge-connected $6$-graph.

\begin{figure}[htbp]
	\centering
	
	\includegraphics[scale=0.45]{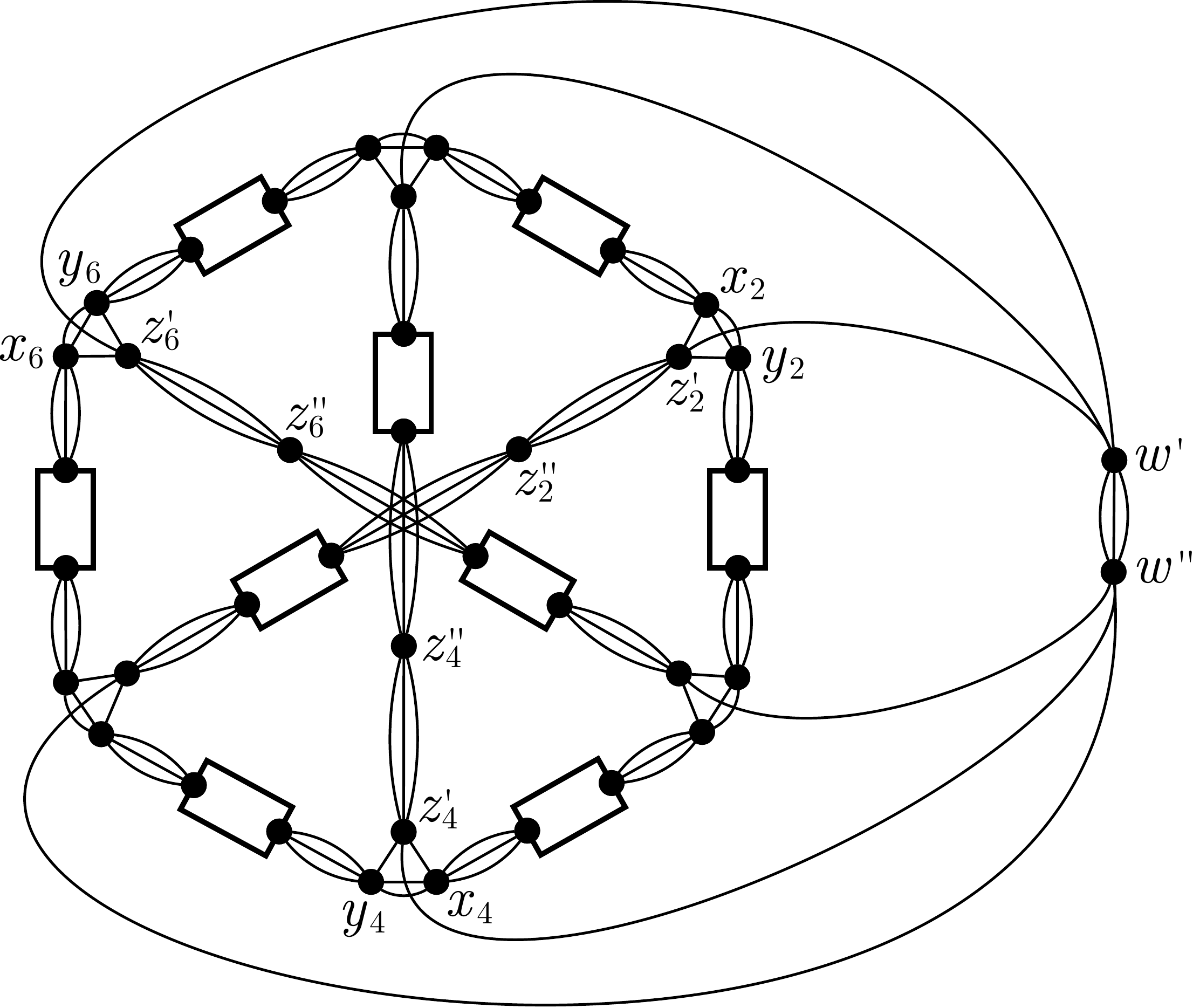}
	\caption{The graph $G^6$.}\label{fig:G^6}
\end{figure}

\begin{prop}
	The graph $G^6$ has no $4$-PDPM.
\end{prop}

\begin{proof}
	In this proof vertex labelings of $G^6$ are considered with reference to Figure \ref{fig:G^6}.
	Assume by contradiction that $G^6$ has a $4$-PDPM $\ca M = \{N_1,\dots,N_4\}$. Then, there is $j\in\{1,\dots,4\}$ such that  $\partial_{G^6}(\{w',w''\})\cap N_j \ne \emptyset.$ Let $e\in \partial_{G^6}(\{w',w''\})\cap N_j$. We can assume without loss of generality that $e$ is incident with $z_2'$. Let $X=\{x_2,y_2,z_2'\}\subseteq V(G^6)$. Then, from Lemma \ref{lem:Q_1}, we infer that $|\partial_{G^6}(X) \cap N|$ is odd, where $N=\cup_{i=1}^{4} N_i$. On the other hand, since $X$ is an odd set, we have that for every $i\in\{i,\dots, 4\}$, $|X\cap N_i|$ is an odd number. Thus, $|X\cap N|= \sum_{i=1}^{4} |X\cap N_i|$ must be an even number, a contradiction.
\end{proof}

\subsubsection*{Base graphs if $l \ge 4$.}

Let $l \ge 4$ and consider the following graph 
\begin{align*}
G^{3 l-4}= P + (l-2)M_0 + (l-3)M_1 + (l-3)M_2 + M_3.
\end{align*}	
The graph $G^8$ is shown in the left-hand side of Figure~\ref{fig:G_8_G_9}. By definition, $G^{3 l-4}$ is a $(3 l-4)$-graph, which is $2l$-edge-connected by Lemma \ref{Lemma-edge-conn-Peterson}. It is well known, see \cite{Grunewald_Steffen_1999}, that $G^{3 l-4}$ is of class 2 and hence has no $(3 l-5)$-PDPM.

\subsubsection*{Proof of Theorem \ref{theo:main_result}.}

   We prove the statement by induction on $r$. When $l \geq 4$ we choose $G^{3 l-4}$ as base graph (defined above) and we consider the perfect matching $M_0$ of $G^{3 l-4}$. 
	
	Recall that $G^{3 l-4}$ is a $2l$-edge-connected $(3 l-4)$-graph with no $(3 l-5)$-PDPM. Furthermore, for all $uv\in M_0$, $\mu_{G^{3l-4}}(u,v)\ge l-1$. Hence the base case is settled.  Then, the inductive step follows by Lemma \ref{lem:construction} and the statement is proved.
	
	When $l=3$, we again argue by induction on $r$. We choose $G^6$ as base graph. We have already proved that it is a $6$-edge-connected $6$-graph without a $4$-PDPM. Hence, $m(6,6)\le3$.
	
	Let $M^6$ be the perfect matching of $G^6$ defined as follows. Consider the matching consisting of the bold red edges depicted in Figure \ref{fig:G^6_pm}. Extend this matching to a perfect matching of $G^6$ by choosing, for every copy of $Q_1$, the bold red edges depicted in Figure \ref{fig:Q_1}. Note that the chosen set of edges is indeed a perfect matching and each edge of such perfect matching has at least one other parallel edge. This means that the condition on the multiplicities of Lemma \ref{lem:construction} is satisfied, i.e.\ for every edge $uv\in M^6$, $\mu_{G^6}(u,v)\ge 2= l - 1$. Therefore the base step is settled. Again, by Lemma \ref{lem:construction}, the inductive step follows. Then Theorem \ref{theo:main_result} is proved.
\ENDproof

\begin{figure}
	\centering
	\includegraphics[scale=0.45]{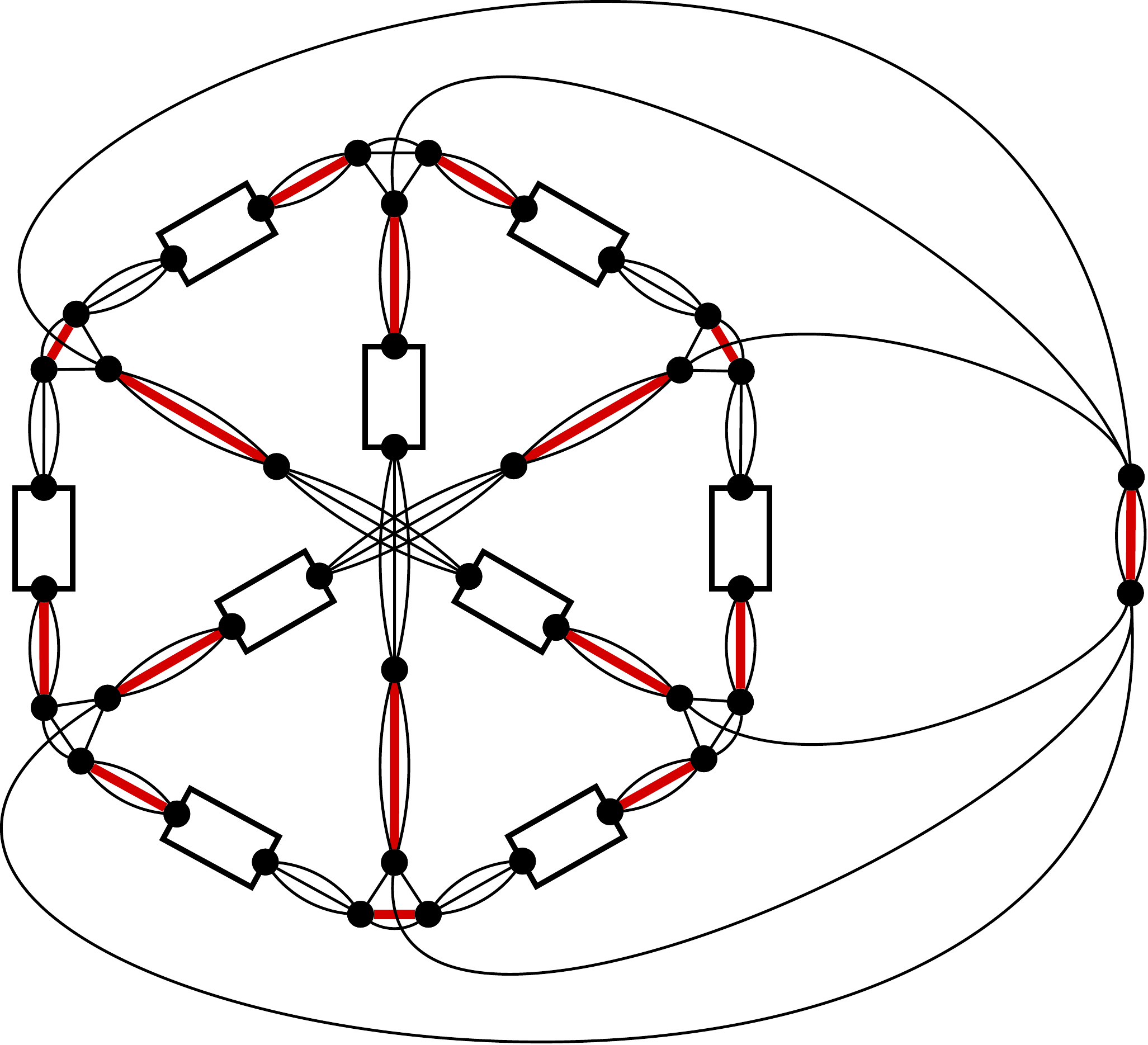}
	\caption{The chosen edges of $G^6$ needed to construct $M^6$.}
	\label{fig:G^6_pm}
\end{figure}

By asking for lower bounds on the parameter $m(t,r)$, one can prove the existence of sets of perfect matchings having specific intersection properties in regular graphs.
For example, it can be proved that for $l \ge 5$, if $m(2l,3l) \ge 2l -1$, then every bridgeless cubic graph admits a perfect matching cover of cardinality $2l -1.$
As another example, it can be proved that, for $l \ge 3$, if $m(2l,3l)\ge l$, then every bridgeless cubic graph has $l$ perfect matchings with empty intersection. Both these proofs rely on the properties of the Petersen graph described in Lemma \ref{Lemma-no-cover-all-edges}.

We though believe that these lower bounds are quite strong conditions. We believe the following statement to be true.

\begin{con}\label{conj:m(2l,r)}
	For all $l\ge 2$ and $r\ge2l$, $m(2l,r)\le l-1$.
\end{con}

Note that when $l=2$, Conjecture \ref{conj:m(2l,r)} is true by Rizzi \cite{rizzi1999indecomposable}.

\section{Acknowledgments}

Major parts of this work were carried out during a stay of Davide Mattiolo at Paderborn University, supported by the Heinrich Hertz-Stiftung.

\bibliography{Lit_reg_graphs}{}
\addcontentsline{toc}{section}{References}
\bibliographystyle{plain}

\end{document}